\documentclass[letterpaper,11pt]{article}

\usepackage{times}

\usepackage{amsthm,amsmath,amsfonts}

\usepackage{graphicx}
\usepackage[small]{caption}
\usepackage[hidelinks]{hyperref}
\newcommand{\orcid}[1]{\,\href{https://orcid.org/#1}{\includegraphics[width=8pt]{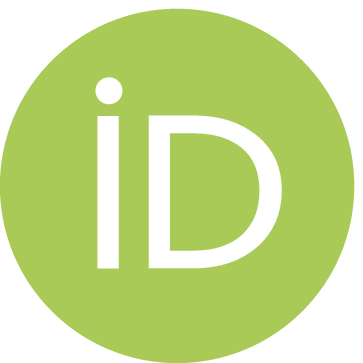}}}

\usepackage{color}

\usepackage{listings}

\newcommand{\C}{\mathcal{C}}
\newcommand{\F}{\mathcal{F}}
\newcommand{\I}{\mathcal{I}}
\newcommand{\J}{\mathcal{J}}

\newcommand{\X}{\mathcal{X}}
\newcommand{\Y}{\mathcal{Y}}
\newcommand{\Z}{\mathcal{Z}}

\newtheorem{theorem}{Theorem}
\newtheorem{corollary}{Corollary}
\newtheorem{lemma}{Lemma}
\newtheorem{proposition}{Proposition}

\title{Partitioning an interval graph into subgraphs with small claws}

\author{Rain Jiang\orcid{0000-0002-0144-942X}\qquad
Kai Jiang\orcid{0000-0001-8165-0571}\qquad
Minghui Jiang\orcid{0000-0003-1843-9292}\,\thanks{\texttt{ dr.minghui.jiang at gmail.com}}\medskip\\
Home School, USA}

\date{}

\begin{document}

\maketitle

\begin{abstract}
The claw number of a graph $G$ is the largest number $v$
such that $K_{1,v}$ is an induced subgraph of $G$.
Interval graphs with claw number at most $v$ are cluster graphs when $v = 1$,
and are proper interval graphs when $v = 2$.

Let $\kappa(n,v)$ be the smallest number $k$ such that
every interval graph with $n$ vertices
admits a vertex partition into $k$ induced subgraphs
with claw number at most $v$.
Let $\check\kappa(w,v)$ be the smallest number $k$ such that
every interval graph with claw number $w$
admits a vertex partition into $k$ induced subgraphs
with claw number at most $v$.
We show that
$\kappa(n,v) = \lfloor\log_{v+1} (n v + 1)\rfloor$,
and that
$\lfloor\log_{v+1} w\rfloor + 1 \le \check\kappa(w,v) \le \lfloor\log_{v+1} w\rfloor + 3$.

Besides the combinatorial bounds,
we also present a simple approximation algorithm for partitioning
an interval graph into the minimum number of induced subgraphs
with claw number at most $v$,
with approximation ratio $3$ when $1 \le v \le 2$, and $2$ when $v \ge 3$.
\end{abstract}

\section{Introduction}

The \emph{intersection graph} $G$ of a family $\F$ of sets is a graph such that
the vertices in $G$ correspond to the sets in $\F$, one vertex for each set,
and an edge connects two vertices in $G$ if and only if the corresponding two
sets in $\F$ intersect.
A family $\F$ of sets is called the \emph{representation} of a graph $G$
if $G$ is the intersection graph of $\F$.

An \emph{interval graph} is the intersection graph of a family of open intervals.
Subclasses of interval graphs can be defined by imposing various restrictions on
the interval representation.
A \emph{proper interval graph} is the intersection graph of a family of open
intervals in which no one properly contains another.
A \emph{unit interval graph} is the intersection graph of a family of open
intervals of the same length.
It is well-known
that an interval graph is a proper interval graph if and only if it is a unit
interval graph, and if and only if it is $K_{1,3}$-free;
see~\cite{BW99} for a self-contained elementary proof.

The \emph{claw number} $\psi(G)$ of a graph $G$ is the largest number
$v \ge 0$ such that $G$ contains the star $K_{1,v}$ as an induced subgraph~\cite{AC10}.
In particular, a graph with claw number $0$ is an empty graph without any edges.
A graph with claw number at most $1$ is simply a disjoint union of cliques,
and is known as a \emph{cluster graph}.
Clearly,
every cluster graph is an interval graph,
with claw number at most $1$.
A proper\,/\,unit interval graph is an interval graph with claw number at most $2$.

In this paper, we study vertex-partitions of an interval graph into
induced subgraphs with small claw numbers.
A \emph{$k$-partition} of a graph
refers to a partition of its vertices into $k$ subsets,
and the resulting $k$ vertex-disjoint induced subgraphs.
For $k \ge 1$ and $v \ge 1$, let $\mu(k,v)$ be the largest number $n$ such
that any interval graph with $n$ vertices admits a vertex partition
into $k$ induced subgraphs with claw number at most $v$.
Correspondingly,
for $n \ge 1$ and $v \ge 1$, let $\kappa(n,v)$ be the smallest number $k$ such
that any interval graph with $n$ vertices admits a vertex partition
into $k$ induced subgraphs with claw number at most $v$.

The \emph{subchromatic number} of a graph
is the smallest number $k$ such that the graph admits a $k$-partition
into cluster graphs~\cite{AJHL89}.
Broersma et~al.~\cite[Lemma~2.4]{BFNW02} proved that the subchromatic number of
any interval graph with $n$ vertices is at most $\lfloor\log_2(n+1)\rfloor$
and this bound is best possible.
Thus $\kappa(n,1) = \lfloor\log_2(n+1)\rfloor$.
Gardi~\cite[Proposition~2.2]{Ga11} showed that for all $n \ge 1$, there exists
an interval graph with $n$ vertices that admits no vertex partition into
less than $\lfloor\log_3(2n+1)\rfloor$ proper interval subgraphs.
Thus $\kappa(n,2) \ge \lfloor\log_3(2n+1)\rfloor$.
The matching upper bound of $\kappa(n,2) \le \lfloor\log_3(2n+1)\rfloor$
was not known.

Extending the two previous results~\cite{BFNW02,Ga11},
we determine exact values of $\mu(k,v)$ and $\kappa(n,v)$ for all $v \ge 1$:

\begin{theorem}\label{thm:mu}
For $k \ge 1$ and $v \ge 1$,
$\mu(k,v) = ((v+1)^{k+1} - (v+1)) / v$.
\end{theorem}

\begin{corollary}\label{cor:kappa}
For $n \ge 1$ and $v \ge 1$,
$\kappa(n,v) = \lfloor\log_{v+1} (n v + 1)\rfloor$.
\end{corollary}

For $k \ge 1$ and $v \ge 1$,
let $\check\mu(k,v)$ be the largest number $w$ such that
any interval graph with claw number at most $w$ admits a vertex partition
into $k$ induced subgraphs with claw number at most $v$.
Correspondingly,
for $w \ge v \ge 1$, let $\check\kappa(w,v)$ be the smallest number $k$ such that
any interval graph with claw number at most $w$ admits a vertex partition
into $k$ induced subgraphs with claw number at most $v$.

Clearly, $\check\mu(k,v) = v$ for $k = 1$ and $v \ge 1$,
and $\check\kappa(w,v) = 1$ for $w = v \ge 1$.
Our next two results include bounds on $\check\mu(k,v)$ for $k \ge 2$ and $v \ge 1$,
and on $\check\kappa(w,v)$ for $w > v \ge 1$:

\begin{theorem}\label{thm:mu'}
For $k \ge 2$ and $v \ge 1$,
$\check\mu(k,v) \le (v+1)^k - 1$.
Also, for $k \ge 2$,
$\check\mu(k,v) \ge (v+1)^{k-1}/2$ when $v = 1$,
$\check\mu(k,v) \ge 2(v+1)^{k-1}/3$ when $v = 2$,
and
$\check\mu(k,v) \ge (v-2)(v+1)^{k-1}$
when $v \ge 3$.
\end{theorem}

\begin{corollary}\label{cor:kappa'}
For $w > v \ge 1$,
$\lfloor\log_{v+1} w\rfloor + 1
\le \check\kappa(w,v)
\le \lfloor\log_{v+1} w\rfloor + 3$.
Moreover, for $w > v$,
\begin{itemize}\setlength\itemsep{0pt}
\item
$\check\kappa(w,v) \le \lfloor\log_{v+1} 2(w-1)\rfloor + 2
\le \lfloor\log_{v+1} w\rfloor + 3$ when $v = 1$,
\item
$\check\kappa(w,v) \le \lfloor\log_{v+1} 3(w-1)/2\rfloor + 2
\le \lfloor\log_{v+1} w\rfloor + 3$ when $v = 2$,
\item
$\check\kappa(w,v) \le \lfloor\log_{v+1} (w-1)/(v-2)\rfloor + 2
\le \lfloor\log_{v+1} w\rfloor + 2$ when $v \ge 3$.
\end{itemize}
\end{corollary}

In the terminology of partially ordered sets~\cite{Fi85},
$\kappa(n,v)$ is the smallest number $k$ such that
any interval order with $n$ elements admits a partition into $k$
weak orders (respectively, semiorders)
when $v=1$ (respectively, $v=2$).
Similarly, $\check\kappa(w,v)$ for $v=1$ and $v=2$ has
alternative interpretations in terms of weak orders and semiorders
in interval orders.

Albertson et~al.~\cite[Theorem~4]{AJHL89} proved that
$\check\kappa(w,1) \le w$ for any $w \ge 1$.
Although this upper bound is asymptotically weaker than the $v=1$ case of
our upper bound in Corollary~\ref{cor:kappa'},
it is tight for $w = 2$ and $3$, as we show in the next theorem,
where the exact values of $\check\kappa(w,v)$ are determined for some small values of $w$ and $v$:

\begin{theorem}\label{thm:small}
$\check\kappa(2,1) = 2$, $\check\kappa(3,1) = 3$,
\quad
$\check\kappa(3,2) = 2$,
$\check\kappa(4,2) = \check\kappa(5,2) = \check\kappa(6,2) = 3$,
\quad
$\check\kappa(5,3) = 2$.
\end{theorem}

For $v \ge 1$,
let \textsc{Min-Partition$(v)$} be the problem of
finding a vertex partition of a given graph
into the minimum number of induced subgraphs with claw number at most $v$.
For $k \ge 2$ and $v \ge 1$,
let \textsc{$k$-Partition$(v)$} be the problem of
deciding whether a given graph admits a vertex partition
into $k$ induced subgraphs with claw number at most $v$.
By certain generic results on vertex-partitioning~\cite{Ac97,Fa04},
\textsc{$k$-Partition$(v)$} in general graphs
is NP-hard for all $k \ge 2$ and $v \ge 1$.

Broersma et~al.~\cite[Theorem~5.4]{BFNW02} proved that for any fixed $k$,
deciding whether an interval graph with $n$ vertices has subchromatic number at most $k$
admits an algorithm running in $O(k\cdot n^{2 k+1})$ time.
In other words,
they presented an $O(k\cdot n^{2 k+1})$-time algorithm for
\textsc{$k$-Partition$(1)$} in interval graphs.

Gandhi et~al.~\cite{GGPR10} noted that
the two results~\cite[Lemma~2.4 and Theorem~5.4]{BFNW02} together imply an
$n^{O(\log n)}$-time exact algorithm for subcoloring interval graphs,
and hence this problem is unlikely to be NP-hard.
Nevertheless,
they presented a $3$-approximation algorithm for
subcoloring interval graphs.
In addition,
they presented a $6$-approximation algorithm for partitioning an interval graph into
the minimum number of proper interval graphs.
In other words, they obtained
a $3$-approximation for \textsc{Min-Partition$(1)$},
and a $6$-approximation for \textsc{Min-Partition$(2)$}, in interval graphs.

We present a simple approximation algorithm for \textsc{Min-Partition$(v)$}
in interval graphs for all $v \ge 1$:

\begin{theorem}\label{thm:approx}
\textsc{Min-Partition$(v)$} in interval graphs
admits a polynomial-time approximation algorithm
with ratio $3$ for $1 \le v \le 2$,
and with ratio $2$ for $v \ge 3$.
\end{theorem}

\section{Preliminaries}

Interval graphs can be recognized in linear time~\cite{Go04}.
Henceforth when we refer to an interval graph,
we assume that an interval representation of the graph is readily available when needed,
and we refer to an interval graph and its interval representation
interchangeably.

An \emph{independent set} (respectively, a \emph{clique})
is a set of pairwise non-adjacent (respectively, adjacent)
vertices in a graph.
Besides the claw number $\psi(G)$ introduced earlier,
there are two other common parameters for a graph $G$:
\begin{itemize}\setlength\itemsep{0pt}

\item
the \emph{independence number}
$\alpha(G)$ is the maximum number of vertices in an independent set
in $G$,

\item
the \emph{vertex-clique-partition number}
$\vartheta(G)$ is the minimum number of parts
in a vertex partition of $G$ into cliques.

\end{itemize}

It is easy to see that
$\psi(G) \le \alpha(G) \le \vartheta(G)$ for any graph $G$.
Indeed, if $G$ is an interval graph, then
$\alpha(G) = \vartheta(G)$
\cite{Go04}.
We next sketch a simple constructive proof of the equality
$\alpha(G) = \vartheta(G)$ for an interval graph $G$.
Let $\I$ be a family of open intervals whose intersection graph is $G$.
Without loss of generality,
assume that all endpoints of intervals in $\I$ are integers.
Then we can find an independent set and a vertex clique partition at the same
time by a standard \emph{sweepline algorithm} as follows:
\begin{quote}
Initialize $\I' \gets \I$ and $i \gets 1$.
While $\I'$ is not empty,
let $T_i = (l_i, r_i)$ be an interval in $\I'$ whose right endpoint $r_i$ is
minimum,
let $\I'_i$ (respectively, $\I_i$) be the subfamily of intervals in
$\I'$ (respectively, $\I$) that contain
the subinterval $S_i = (r_i - 1, r_i)$ of $T_i$,
then update $\I' \gets \I' \setminus \I'_i$ and $i \gets i + 1$.
\end{quote}

Let $j$ be the number of rounds that
the sweepline algorithm runs until $\I'$ is empty.
Then for $1 \le i \le j$,
$S_i \subseteq T_i \in \I'_i \subseteq \I_i$.
Moreover,
the $j$ intervals $T_i$ correspond to an independent set in $G$,
and
the $j$ subfamilies $\I'_i$ correspond to a partition of the vertices of $G$
into cliques.
Thus $\alpha(G) = \vartheta(G)$.

\bigskip
For a graph $G$ and $v \ge 1$,
let $\kappa(G,v)$ be the smallest number of induced subgraphs
in a vertex partition of $G$ such that
each subgraph has claw number at most $v$.

\begin{lemma}\label{lem:vartheta}
For $k \ge 1$ and $v \ge 1$,
any interval graph $G$ with $\vartheta(G) \le (v+1)^k - 1$
satisfies $\kappa(G,v) \le k$.
\end{lemma}

\begin{proof}
Fix any $v \ge 1$.
We prove the lemma by induction on $k$.
For the base case when $k = 1$,
any interval graph $G$ with $\vartheta(G) \le (v+1)^1 - 1 = v$
has $\psi(G) \le \vartheta(G) \le v$ and hence $\kappa(G,v) \le 1$.
Now let $k \ge 2$,
and let $G$ be an interval graph with $\vartheta(G) \le (v+1)^k - 1$.
We next show that $G$ admits a vertex partition into $k$ induced subgraphs
with claw number at most $v$.

Run the sweepline algorithm on an interval representation $\I$ of $G$
to obtain $j = \vartheta(G)$ cliques $\I_i$,
$1 \le i \le j$,
such that $\I = \cup_i\, \I_i$.
Let $t = (v+1)^k - 1$ and $s = (v+1)^{k-1} - 1$.
Then $j \le t = (v+1)(s+1) - 1 = v(s+1) + s$.
Let $\J$ be the union of all subfamilies $\I_i$ with
$i \bmod (s + 1) = 0$.
The subgraph of $G$ represented by $\J$
has vertex-clique-partition number at most $v$,
and hence claw number at most $v$ too.

Each connected component in the subgraph of $G$ represented by $\I \setminus \J$
has vertex-clique-partition number at most $s = (v+1)^{k-1} - 1$,
and hence admits a vertex partition into $k-1$ induced subgraphs
with claw number at most $v$, by the induction hypothesis.
Then by disjoint union,
the subgraph of $G$ represented by $\I \setminus \J$
also admits a vertex partition into $k-1$ induced subgraphs
with claw number at most $v$.
Thus $\kappa(G,v) \le 1 + (k - 1) = k$.
\end{proof}

To prove that their upper bound on subchromatic numbers of interval graphs
is best possible,
Broersma et~al.~\cite[Lemma~2.4]{BFNW02} showed that for any $k \ge 1$,
there is an interval graph $G_k$ with $2^k - 1$ vertices
and subchromatic number $k$.
Such graphs were later used by Gandhi et~al.~\cite{GGPR10} as a lower bound in
their $3$-approximation algorithm for subcoloring interval graphs;
they refer to $G_k$ as $BC(k)$ and call them binary cliques.
Similarly,
Gardi~\cite[Lemma~2.1]{Ga11} showed that for any $k \ge 1$,
there is an interval graph $H_k$ with $(3^k - 1) / 2$ vertices
that admits no vertex partition into less than $k$ proper interval graphs.

$G_k$ and $H_k$ for $k \ge 1$ are recursively constructed as follows:
$G_1$ (respectively, $H_1$) is just a single vertex.
For $k \ge 2$, $G_k$ (respectively, $H_k$) consists of a new vertex
connected to all vertices of $2$ (respectively, $3$) disjoint copies of
$G_{k-1}$ (respectively, $H_{k-1}$).

$G_k$ and $H_k$ for $k \ge 1$ can be generalized to
$A_{k,v}$ for $k \ge 1$ and $v \ge 1$:
$A_{1,v}$ is just a single vertex.
For $k \ge 2$, $A_{k,v}$ consists of a new vertex
connected to all vertices of $v+1$ disjoint copies of $A_{k-1,v}$.
Then $G_k = A_{k,1}$
and $H_k = A_{k,2}$.
Refer to Figure~\ref{fig:h3} for an illustration of $H_3 = A_{3,2}$.

\begin{figure}[htbp]
\centering\includegraphics{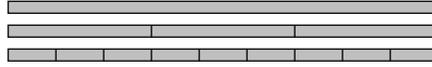}
\caption{An interval representation of $H_3 = A_{3,2}$.}
\label{fig:h3}
\end{figure}

It is easy to see that for $k \ge 1$ and $v \ge 1$,
$A_{k,v}$ is an interval graph with
$\sum_{i=0}^{k-1} (v+1)^i = ((v+1)^k - 1)/v$ vertices,
and
$\alpha(A_{k,v}) = \vartheta(A_{k,v}) = (v+1)^{k-1}$.
Also, for $k \ge 2$ and $v \ge 1$,
$\psi(A_{k,v}) = (v+1)^{k-1}$.

\begin{lemma}\label{lem:A}
For $k \ge 2$ and $v \ge 1$,
$\kappa(A_{k,v},v) = k$.
\end{lemma}

\begin{proof}
Fix any $v \ge 1$.
Since
$\vartheta(A_{k,v}) = (v+1)^{k-1} \le (v+1)^k - 1$,
it follows by Lemma~\ref{lem:vartheta} that
$\kappa(A_{k,v},v) \le k$.
We next show that
$\kappa(A_{k,v},v) \ge k$ by induction on $k$.
For the base case when $k = 2$,
$A_{2,v}$ is simply $K_{1,v+1}$, and hence admits no vertex partition
into less than $2$ induced subgraphs with claw number at most~$v$.
Now proceed to the inductive step when $k \ge 3$.

Consider any vertex partition of $A_{k,v}$ into
induced subgraphs with claw number at most $v$.
Recall that $A_{k,v}$ consists of a new vertex
connected to $v+1$ disjoint copies of $A_{k-1,v}$.
To avoid a star $K_{1,v+1}$ forming around the new vertex,
the subgraph that includes it can include
vertices from at most $v$ copies of $A_{k-1,v}$,
and hence must miss one copy of $A_{k-1,v}$ entirely.
By the induction hypothesis,
this copy of $A_{k-1,v}$ admits no vertex partition
into less than $k-1$ induced subgraphs with claw number at most $v$.
It follows that $A_{k,v}$ admits no vertex partition
into less than $k$ induced subgraphs with claw number at most $v$.
Thus $\kappa(A_{k,v},v) \ge k$.
\end{proof}

Note that for $k \ge 1$ and $v \ge 1$,
the interval graph $A_{k+1,v}$ satisfies
$\vartheta(A_{k+1,v}) = (v+1)^k$
and
$\kappa(A_{k+1,v},v) = k+1$.
This shows that the bound in Lemma~\ref{lem:vartheta} is best possible.

\section{$\mu(k,v)$ and $\kappa(n,v)$}

In this section we prove Theorem~\ref{thm:mu} and
Corollary~\ref{cor:kappa}.

For a family $\I$ of intervals, denote by $|\I|$ the number of intervals in $\I$.
We first prove a technical lemma:

\begin{lemma}\label{lem:2r+1}
For $n \ge s \ge 1$,
the vertex set of any interval graph $G$ with $n$ vertices
can be partitioned into $2r+1$ subsets
for some $r \ge 0$,
including $X_i$ and $Y_i$ for $1 \le i \le r$, and $Z$,
such that
\begin{enumerate}\setlength\itemsep{0pt}

\item
Vertices from different subsets among $X_i$ and $Z$ are non-adjacent in $G$.

\item
Vertices in each subset $Y_i$ are pairwise adjacent in $G$.

\item
$|Z| \le |X_i| = s < |X_i| + |Y_i|$.

\end{enumerate}
\end{lemma}

\begin{proof}
For any family $\Z$ of intervals, and for $p \le q$,
denote by $\Z(p,q)$ the subfamily of intervals in $\Z$
that are contained in the interval $(p,q)$.

Obtain an interval representation $\I$ of $G$ with integer endpoints.
Let $a$ be the leftmost endpoint, and $c$ the rightmost endpoint, of the intervals in $\I$.

Initialize $\Z \gets \I$ and $i \gets 1$.
While $|\Z| > s$,
let $b$ be the smallest integer in $(a,c]$ such that
$|\Z(a,b-1)| \le s < |\Z(a,b)|$,
let $\X_i$ be any subfamily of $s$ intervals in $\Z$ such that
$\Z(a,b-1) \subseteq \X_i \subset \Z(a,b)$,
let $\Y_i$ be $\Z \setminus (\X_i \cup \Z(b,c))$,
then update $\Z \gets \Z \setminus (\X_i \cup \Y_i)$ and $i \gets i+1$.

Let $r \ge 0$ be the number of such partitioning rounds until $|\Z| \le s$.
Then $\I$ is partitioned into $2r+1$ subfamilies,
including $\X_i$ and $\Y_i$ for $1 \le i \le r$,
and a (possibly empty) subfamily $\Z$ in the end.
The following three properties can be easily verified:
\begin{enumerate}\setlength\itemsep{0pt}

\item
Intervals from different subfamilies among $\X_i$ and $\Z$ do not intersect.

\item
Intervals in each subfamily $\Y_i$ pairwise intersect.

\item
$|\Z| \le |\X_i| = s < |\X_i| + |\Y_i|$.

\end{enumerate}
Let $X_i$, $Y_i$, and $Z$ be the subsets of vertices of $G$ represented by
intervals in $\X_i$, $\Y_i$, and $\Z$, respectively.
Then the proof is complete.
\end{proof}

We now prove Theorem~\ref{thm:mu} that
$\mu(k,v) = ((v+1)^{k+1} - (v+1))/v$
for $k \ge 1$ and $v \ge 1$.
Fix $v \ge 1$,
and	let $n_k = ((v+1)^{k+1} - (v+1)) / v$.

Recall Lemma~\ref{lem:A} that for $k \ge 1$ and $v \ge 1$,
the interval graph $A_{k+1,v}$ has
$((v+1)^{k+1} - 1) / v$ vertices,
and satisfies
$\kappa(A_{k+1,v},v) = k+1$.
This implies the upper bound of
$\mu(k,v) \le ((v+1)^{k+1} - 1) / v - 1 = n_k$.
In the following we prove the matching lower bound of
$\mu(k,v) \ge n_k$ by induction on $k$.

For the base case when $k = 1$,
we have $n_1 = ((v+1)^{1+1} - (v+1)) / v = v+1$.
Any interval graph with at most $v+1$ vertices
obviously has claw number at most $v$.
Thus $\mu(1,v) \ge n_1$.
Now proceed to the inductive step when $k \ge 2$,
and let $G$ be an interval graph with at most $n_k$ vertices.

Apply Lemma~\ref{lem:2r+1} with $s = n_{k-1}$ to partition the vertex set
of $G$ into $2r+1$ subsets, $X_i$ and $Y_i$ for $1 \le i \le r$, and $Z$.
Since $n_k = (v+1)(n_{k-1} + 1) = (v+1)(s + 1)$,
we have $1 \le r \le v + 1$.
Consider two cases:
\begin{itemize}\setlength\itemsep{0pt}

\item $r \le v$.
Note that $|X_i| = n_{k-1}$ for $1 \le i \le r$, and that $|Z| \le n_{k-1}$.
By the induction hypothesis,
any interval graph with at most $n_{k-1}$ vertices
admits a vertex partition into $k-1$ induced subgraphs
with claw number at most $v$.
Then by disjoint union,
the subgraph of $G$ induced by all vertices
in the $r$ subsets $X_i$ and the subset $Z$
admits a vertex partition into $k-1$ subgraphs
with claw number at most $v$.
On the other hand, the subgraph of $G$ induced by vertices
in the $r$ subsets $Y_i$
is the union of $r$ cliques,
and has vertex-clique-partition number at most $r$,
and hence has claw number at most $r \le v$.

\item $r = v + 1$.
Then we must have $|X_i| = n_{k-1}$ and $|Y_i| = 1$ for $1 \le i \le v+1$, and $|Z| = 0$.
The subgraph of $G$ induced by all vertices in the $r$ subsets $X_i$
admits a vertex partition into $k-1$ subgraphs
with claw number at most $v$.
On the other hand, the subgraph of $G$ induced by the $v+1$ vertices
in the $v+1$ subsets $Y_i$
clearly has claw number at most $v$ too.

\end{itemize}
In both cases, $G$ admits a vertex partition into $k$ induced subgraphs
with claw number at most $v$.
This completes the proof of Theorem~\ref{thm:mu}.

\bigskip
We next prove Corollary~\ref{cor:kappa} that
$\kappa(n, v) = \lfloor\log_{v+1}(n v + 1)\rfloor$ for $n \ge 1$ and $v \ge 1$.
Fix any $v \ge 1$.

For $1 \le n \le v+1$,
we have $v + 1 \le n v + 1 < (v+1)^2$,
and hence
$\lfloor\log_{v+1}(n v + 1)\rfloor = 1$.
It is clear that $\kappa(n,v) = 1$ for this case.

Now fix $n \ge v+2$.
Then $n v + 1 \ge (v+1)^2$
and hence
$\lfloor\log_{v+1} (n v + 1)\rfloor \ge 2$.
Let $k = \lfloor\log_{v+1} (n v + 1)\rfloor$.
Then $n \ge ((v+1)^k - 1) / v$.
By Lemma~\ref{lem:A},
the interval graph $A_{k,v}$ with $((v+1)^k - 1) / v$ vertices
satisfies
$\kappa(A_{k,v}, v) = k$.
Thus
$\kappa(n,v) \ge \kappa(((v+1)^k - 1) / v, v) \ge k$.

For the other direction,
let $\kappa = \kappa(n,v)$.
Then by definition of $\mu$, we have $n \ge \mu(\kappa-1,v) + 1$.
By Theorem~\ref{thm:mu},
$\mu(\kappa-1,v) = ((v+1)^\kappa - (v+1)) / v$.
Thus
$n \ge ((v+1)^\kappa - 1) / v$.
It follows that
$\kappa \le \lfloor\log_{v+1} (n v + 1)\rfloor = k$.
This completes the proof of Corollary~\ref{cor:kappa}.

\section{$\check\mu(k,v)$ and $\check\kappa(w,v)$}

In this section we prove
Theorem~\ref{thm:mu'} and Corollary~\ref{cor:kappa'}.

We first prove Theorem~\ref{thm:mu'}.
Fix $k \ge 2$ and $v \ge 1$.
Recall Lemma~\ref{lem:A} that
the interval graph $A_{k+1,v}$ with claw number $(v+1)^k$
satisfies
$\kappa(A_{k+1,v},v) = k+1$.
This implies the upper bound of
$\check\mu(k,v) \le (v+1)^k - 1$.

Write
$w_1 = (v+1)^{k-1}/2$,
$w_2 = 2(v+1)^{k-1}/3$,
and
$w_3 = (v-2)(v+1)^{k-1}$.
Let $G$ be an interval graph with claw number
$w \le w_1$ when $v = 1$,
$w \le w_2$ when $v = 2$,
and $w \le w_3$ when $v \ge 3$.
In the following,
we obtain a vertex partition of $G$ into $k$ induced subgraphs
with claw number at most $v$,
hence proving the lower bounds
$\check\mu(k,v) \ge w_1$ when $v = 1$,
$\check\mu(k,v) \ge w_2$ when $v = 2$,
and
$\check\mu(k,v) \ge w_3$ when $v \ge 3$.

Let $\I$ be an interval representation of $G$ with integer endpoints.
Let $j = \vartheta(G)$.
Run the sweepline algorithm to obtain an independent set of $j$ intervals $T_i$
in $\I$,
and a vertex covering of $\I$ by $j$ cliques $\I_i$,
where $T_i \in \I_i$ for $1 \le i \le j$.

Let $s = 2w$ when $v = 1$,
$s = \lceil 3w/2 \rceil$ when $v = 2$,
and $s = \lceil w/(v-2) \rceil$ when $v \ge 3$.
Then for all $v \ge 1$,
$s \le (v+1)^{k-1}$.
Let $\J$ be the union of all subfamilies $\I_i$ with
$i \bmod s = 0$.
Then each connected component in the graph represented by $\I \setminus \J$
has vertex-clique-partition number at most $s - 1 \le (v+1)^{k-1} - 1$,
and hence admits a vertex partition into
$k-1$ induced subgraphs with claw number at most $v$,
by Lemma~\ref{lem:vartheta}.
Then by disjoint union,
the graph represented $\I \setminus \J$
also admits a vertex partition into
$k-1$ induced subgraphs with claw number at most $v$.
It remains to show that the graph represented by $\J$ has claw number at most $v$.

Suppose for contradiction that
the graph represented by $\J$ has claw number at least $v+1$.
Let $C$ and $L_1,\ldots,L_{v+1}$ be intervals representing the center and the
$v+1$ leaves of a star $K_{1,v+1}$ in the graph,
where the $v+1$ intervals $L_1,\ldots,L_{v+1}$ are ordered from left to right.
We now proceed in three different ways depending on the value of $v$:
\begin{enumerate}\setlength\itemsep{0pt}

\item $v = 1$.
Either $L_1$ and $C$, or $C$ and $L_2$, are two intervals from
two subfamilies $\I_{ps}$ and $\I_{qs}$ with $q - p \ge 1$.
The union of the two intervals
is a contiguous interval that intersects all intervals in
$qs-ps+1$ consecutive subfamilies $\I_i$, $ps \le i \le qs$,
which include
the $qs-ps+1$ disjoint intervals $T_i$, $ps \le i \le qs$.
Then one of these two intervals intersects at least
$\lceil (qs-ps+1)/2 \rceil
\ge \lceil (s+1)/2 \rceil = \lceil (2w+1)/2 \rceil = w + 1$
disjoint intervals $T_i$.

\item $v = 2$.
$L_1$ and $L_3$ are from
two subfamilies $\I_{ps}$ and $\I_{qs}$ with $q - p \ge 2$.
The union of $L_1$, $C$, and $L_3$
is a contiguous interval that intersects all intervals in
$qs-ps+1$ consecutive subfamilies $\I_i$, $ps \le i \le qs$.
One of these three intervals intersects at least
$\lceil (qs-ps+1)/3 \rceil
\ge \lceil (2s+1)/3 \rceil \ge \lceil (3w+1)/3 \rceil = w + 1$
disjoint intervals $T_i$.

\item $v \ge 3$.
$L_2$ and $L_v$ are from
two subfamilies $\I_{ps}$ and $\I_{qs}$ with $q - p \ge v - 2$.
$C$ intersects all intervals in
$qs-ps+1$ consecutive subfamilies $\I_i$, $ps \le i \le qs$.
Note that
$qs - ps + 1 \ge (v - 2)s + 1 \ge w + 1$.
Thus $C$ intersects at least $w + 1$ disjoint intervals $T_i$.

\end{enumerate}
In each of the three cases,
we can find a star $K_{1,w+1}$ in $G$ represented by
some interval in $\J$ intersecting
$w+1$ disjoint intervals $T_i \in \I_i$,
a contradiction to our assumption that $G$ has claw number $w$.
Thus the graph represented by $\J$ must have claw number at most $v$.
This completes the proof of Theorem~\ref{thm:mu'}.

\bigskip
We next prove the bounds on
$\check\kappa(w,v)$ in Corollary~\ref{cor:kappa'}
for $w > v \ge 1$.
Let $k = \lfloor\log_{v+1} w\rfloor + 1$.
Then $w \ge (v+1)^{k-1}$ and $k \ge 2$.
By Lemma~\ref{lem:A},
the interval graph $A_{k,v}$ with claw number $(v+1)^{k-1}$
satisfies $\kappa(A_{k,v},v) = k$.
Thus we have the lower bound
$\check\kappa(w,v) \ge \check\kappa((v+1)^{k-1},v) \ge k$.

For the other direction,
let $\check\kappa = \check\kappa(w,v)$, where $\check\kappa \ge 2$.
Then by definition of $\check\mu$, we have $w \ge \check\mu(\check\kappa-1,v) + 1$.
By the lower bounds in Theorem~\ref{thm:mu'},
we have
$w \ge (v+1)^{\check\kappa-2}/2 + 1$ when $v = 1$,
$w \ge 2(v+1)^{\check\kappa-2}/3 + 1$ when $v = 2$,
and
$w \ge (v-2)(v+1)^{\check\kappa-2} + 1$ when $v \ge 3$.
Correspondingly,
\begin{itemize}\setlength\itemsep{0pt}

\item
when $v = 1$,\,
$\check\kappa(w, v)
\le \lfloor \log_{v+1}{2(w-1)} \rfloor + 2
= \lfloor \log_{v+1}{(w-1)} \rfloor + 3
\le \lfloor \log_{v+1}{w} \rfloor + 3$,

\item
when $v = 2$,\,
$\check\kappa(w, v)
\le \lfloor \log_{v+1}{3(w-1)/2} \rfloor + 2
= \lfloor \log_{v+1}{(w-1)/2} \rfloor + 3
\le \lfloor \log_{v+1}{w} \rfloor + 3$,

\item
when $v \ge 3$,\,
$\check\kappa(w, v)
\le \lfloor \log_{v+1}{(w-1)/(v-2)} \rfloor + 2
\le \lfloor \log_{v+1}{w} \rfloor + 2$.

\end{itemize}
Thus we have the upper bounds on $\check\kappa(w,v)$.
This completes the proof of Corollary~\ref{cor:kappa'}.

\section{$\check\kappa(w,v)$ for small $w$ and $v$}

In this section we prove Theorem~\ref{thm:small} that
$\check\kappa(2,1) = 2$, $\check\kappa(3,1) = 3$,
$\check\kappa(3,2) = 2$,
$\check\kappa(4,2) = \check\kappa(5,2) = \check\kappa(6,2) = 3$,
and
$\check\kappa(5,3) = 2$.

We first prove that $\check\kappa(2,1) = 2$ and $\check\kappa(3,1) = 3$.
Recall~\cite[Theorem~4]{AJHL89} that
$\check\kappa(w,1) \le w$ for any $w \ge 1$.
In particular,
$\check\kappa(2,1) \le 2$ and $\check\kappa(3,1) \le 3$.
From the other direction, it is clear that
$\check\kappa(w,v) \ge 2$ for any $w > v \ge 1$.
Thus $\check\kappa(2,1) = 2$.
The following lemma implies that $\check\kappa(3,1) \ge 3$,
and hence we have $\check\kappa(3,1) = 3$.

\begin{lemma}\label{lem:w3v1}
Let $\J_3$ be the set of $12$ open intervals with lengths $1$, $2$, and $3$,
and with integer endpoints between $0$ and $5$.
Then the interval graph represented by $\J_3$ has claw number $3$
and admits a vertex partition into three subgraphs
with claw number at most $1$,
but it admits no vertex partition into two subgraphs
with claw number at most $1$.
\end{lemma}

\begin{proof}
The interval graph represented by $\J_3$,
with maximum interval length $3$,
clearly has claw number $3$.
Moreover, it admits a vertex partition, by left endpoint modulo $3$,
into three subgraphs with claw number at most $1$ (that is, three cluster graphs).
To show that the graph admits no vertex partition into two subgraphs
with claw number at most $1$,
we suppose the contrary and color each interval in $\J_3$ either white or black
according to such a partition,
then find three intervals of the same color that represent $K_{1,2}$ to reach
a contradiction.

\begin{figure}[htbp]
\centering\includegraphics{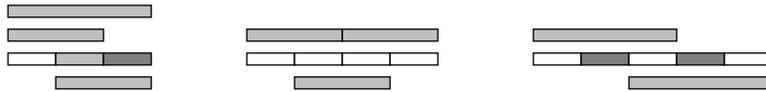}
\caption{Intervals from $\J_3$ for the three cases in the proof of Lemma~\ref{lem:w3v1}.}
\label{fig:w3v1}
\end{figure}

Refer to Figure~\ref{fig:w3v1}.
Consider three cases:
\begin{enumerate}\setlength\itemsep{0pt}
\item
$\J_3$ contains three consecutive unit intervals among which
the left and the right have opposite colors.
Let $(a,a+1)$, $(a+1,a+2)$, $(a+2,a+3)$ be the three unit intervals.
Consider also the two length-$2$ intervals $(a,a+2)$ and $(a+1,a+3)$,
and the length-$3$ interval $(a,a+3)$.
By symmetry,
we can assume without loss of generality that
$(a,a+1)$ and $(a+1,a+2)$ are white, and $(a+2,a+3)$ is black.
Then both $(a,a+2)$ and $(a,a+3)$ must be black to avoid forming a white star $K_{1,2}$
with $(a,a+1)$ and $(a+1,a+2)$.
But then $(a,a+3)$, $(a,a+2)$, $(a+2,a+3)$ form a black star $K_{1,2}$.
\item
The five unit intervals in $\J_3$ all have the same color.
Consider any four consecutive unit intervals,
say, $(a,a+1)$, $(a+1,a+2)$, $(a+2,a+3)$, $(a+3,a+4)$.
Assume without loss of generality that they are all white.
Consider also the three length-$2$ intervals $(a,a+2)$, $(a+1,a+3)$, $(a+2,a+4)$.
Either one of them is white,
forming a white star $K_{1,2}$ with two unit intervals,
or all of them are black, forming a black star $K_{1,2}$ by themselves.
\item
The five unit intervals in $\J_3$ have alternating colors.
Let $(a,a+1)$, $(a+1,a+2)$, $(a+2,a+3)$, $(a+3,a+4)$, $(a+4,a+5)$
be the five unit intervals.
Assume without loss of generality that
$(a,a+1)$, $(a+2,a+3)$, $(a+4,a+5)$ are white,
$(a+1,a+2)$ and $(a+3,a+4)$ are black.
Then the length-$3$ interval $(a,a+3)$ must be black
to avoid forming a white star $K_{1,2}$ with $(a,a+1)$ and $(a+2,a+3)$.
Similarly,
the length-$3$ interval $(a+2,a+5)$ must be black
to avoid forming a white star $K_{1,2}$ with $(a+2,a+3)$ and $(a+4,a+5)$.
Then the two intervals $(a,a+3)$ and $(a+2,a+5)$, together with either 
$(a+1,a+2)$ or $(a+3,a+4)$,
form a black star $K_{1,2}$.
\end{enumerate}
In each case, there are three intervals in $\J_3$
representing a monochromatic star $K_{1,2}$.
\end{proof}

We next prove that
$\check\kappa(3,2) = 2$
and
$\check\kappa(4,2) = \check\kappa(5,2) = \check\kappa(6,2) = 3$.
By our upper bound of
$\check\kappa(w,v) \le \lfloor \log_{v+1} 3(w-1)/2 \rfloor + 2$
for $w > v = 2$ in Corollary~\ref{cor:kappa'},
we have $\check\kappa(6,2) \le 3$.
On the other hand, it is clear that
$\check\kappa(3,2) \ge 2$.
Thus
\begin{equation}\label{eq:v2}
2 \le \check\kappa(3,2) \le \check\kappa(4,2) \le \check\kappa(5,2) \le \check\kappa(6,2) \le 3.
\end{equation}

As a warm-up exercise,
we first prove the following lemma
which implies that $\check\kappa(6,2) \ge 3$,
and hence $\check\kappa(6,2) = 3$.

\begin{lemma}\label{lem:w6v2}
Let $\J_6$ be the family of $24$ open intervals with lengths $1$, $3$, $5$,
and $6$ as illustrated in Figure~\ref{fig:w6v2}.
Then the interval graph represented by $\J_6$ has claw number $6$
and admits a vertex partition into three subgraphs
with claw number at most $2$,
but it admits no vertex partition into two subgraphs
with claw number at most $2$.
\end{lemma}

\begin{figure}[htbp]
\centering\includegraphics{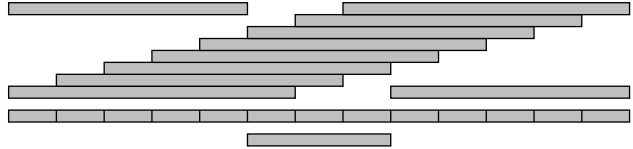}
\caption{The family $\J_6$ of $24$ intervals includes
$13$, $1$, $2$, and $8$ intervals with lengths $1$, $3$, $5$, and $6$,
respectively.}
\label{fig:w6v2}
\end{figure}

\begin{proof}
The interval graph represented by $\J_6$ clearly has claw number $6$,
and admits a vertex partition, by lengths $\{1\} \cup \{3\} \cup \{5,6\}$,
into three subgraphs with claw number at most $2$.
Suppose for contradiction that it admits a vertex partition into two
subgraphs with claw number at most $2$.
Then any two intervals in $\J_6$ containing five unit intervals in their
intersection must have the same color,
because otherwise one of them would have the same color as three of the five unit intervals,
forming a monochromatic star $K_{1,3}$.
Consequently all length-$5$ and length-$6$ intervals must have the same color.

Consider the length-$3$ interval in the middle
and the three unit intervals contained in it.
If any of these four intervals has the same color as the length-$5$ and length-$6$ intervals,
then it would form a monochromatic star together with the two length-$5$ intervals
and a length-$6$ interval that intersects all of them.
Otherwise, these four intervals in the middle would have the same color,
and would form a monochromatic star $K_{1,3}$ by themselves.
In both cases, we reach a contradiction.
\end{proof}

The next lemma implies that $\check\kappa(5,2) \ge 3$,
and hence $\check\kappa(5,2) = 3$.

\begin{lemma}\label{lem:w5v2}
Let $\J_5$ be the family of open intervals with lengths $1$, $3$, and $5$,
and with integer endpoints between $0$ and $79$.
The interval graph represented by $\J_5$ has claw number $5$
and admits a vertex partition into three subgraphs
with claw number at most $2$,
but it admits no vertex partition into two subgraphs
with claw number at most $2$.
\end{lemma}

\begin{proof}
The interval graph represented by $\J_5$ clearly has claw number $5$,
and admits a vertex partition (by lengths) into three subgraphs
with claw number at most $2$.
To show that the graph admits no vertex partition into two subgraphs
with claw number at most $2$,
we suppose the contrary and color each interval in $\J_5$ either white or black
according to such a partition,
then find four intervals of the same color that represent a star $K_{1,3}$
to reach a contradiction.

\begin{figure}[htbp]
\centering\includegraphics{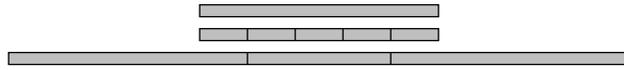}
\caption{A length-$5$ interval and eight neighbors in $\J_5$.}
\label{fig:w5v2}
\end{figure}

Consider any subset of nine intervals in $\J_5$ whose lengths and relative
positions are as illustrated in Figure~\ref{fig:w5v2}.
Focus on the length-$5$ interval in the middle,
and consider the other eight intervals as its neighbors.
We refer to the length-$5$ interval on the left
and the leftmost unit interval as the \emph{left neighbors},
the length-$5$ interval on the right
and the rightmost unit interval as the \emph{right neighbors},
the length-$3$ interval and the three unit intervals
in the middle as the \emph{middle neighbors},
respectively,
of the length-$5$ interval in the middle.
If the length-$5$ interval in the middle has the same color as
both a left neighbor and a right neighbor,
then to avoid a monochromatic star $K_{1,3}$ it cannot have the same color as
any of its middle neighbors.
But then the four middle neighbors would have the same color,
and form a monochromatic star $K_{1,3}$ by themselves.
Thus we have the following property:
\begin{quote}
No length-$5$ interval in $\J_5$ can have
both a left neighbor and a right neighbor of the same color as itself.
\end{quote}

Partition the unit intervals in $\J_5$ into four \emph{groups}
according to their left endpoints modulo $4$.
We say that two unit intervals in the same group are
\emph{$4$-adjacent} if their left endpoints differ by exactly $4$.
For any two $4$-adjacent unit intervals $A$ and $B$ in $\J_5$,
denote by $AB$ the unique length-$5$ interval in $\J_5$ that contains them.

\begin{figure}[htbp]
\centering\includegraphics{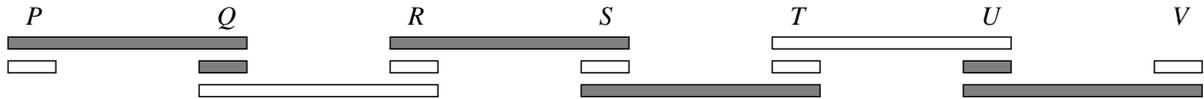}
\caption{A chain of consecutive $4$-adjacent unit intervals
$P,Q,R,S,T,U,V$.}
\label{fig:w5v2chain}
\end{figure}

Consider a chain of $4$-adjacent unit intervals from the same group in $\J_5$,
in the middle row of the three rows of intervals illustrated in Figure~\ref{fig:w5v2chain}.
Suppose that both $R$ and $S$ are white.
Then $RS$ must be black.
Then $QR$ and $ST$ cannot be both black.
Without loss of generality, assume that $QR$ is white.
Then $Q$ is black, and so is $PQ$.
Then the argument repeats
as the chain extends further to the left with alternating colors:
$P$ is white,
the unit interval that is $4$-adjacent to $P$ on the left, if any, is
black, and so on.
Now consider $ST$.
If $ST$ is white, then the situation is symmetric to the case that
$QR$ is white,
with respect to the starting interval $RS$.
If $ST$ is black, then $T$ is white and so is $TU$.
Then the argument repeats
as the chain extends further to the right with alternating colors,
for $U$ and $V$ and so on.

In summary,
in the chain of $4$-adjacent unit intervals from each group,
there is at most one monochromatic subchain
of two or more consecutive $4$-adjacent unit intervals
of the same color,
and moreover the length of such a monochromatic subchain, if it exists,
is either two or three.
For each such monochromatic subchain, if its length is two, mark any one
of the two unit intervals as a hole,
or else if its length is three, mark the unit interval in
the middle (for example, $S$ among $R,S,T$ as in Figure~\ref{fig:w5v2chain}) as a hole.
Then there are at most $h \le 4$ holes from the four groups,
and they separate the other $79 - h$ unit intervals in $\J_5$ into
at most $h+1$ contiguous \emph{blocks}.
The longest block has length at least
$(79-4)/(4+1) = 15$.
Within each block, every two $4$-adjacent unit intervals have
different colors.

Encode white as $0$ and black as $1$.
Then the colors of all unit intervals in each block form a periodic
binary string of period $8$,
whose periodic pattern is the concatenation of a $4$-bit binary string
and its complement.
There are $16$ periodic patterns:
\begin{gather*}
00001111
\quad
00011110
\quad
00101101
\quad
00111100
\quad
01001011
\quad
01011010
\quad
01101001
\quad
01111000
\\
10000111
\quad
10010110
\quad
10100101
\quad
10110100
\quad
11000011
\quad
11010010
\quad
11100001
\quad
11110000
\end{gather*}
It is easy to verify that
each of the $16$ patterns is a rotation of either $10100101$ or $11000011$.
Since the longest block has length at least $15 = 2\cdot 8 - 1$,
the binary string corresponding to this block must contain 
either $10100101$ or $11000011$ as a substring.

\begin{figure}[htbp]
\centering\includegraphics{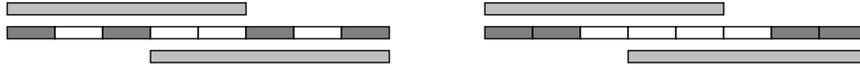}
\caption{The colors of all unit intervals in each block form a periodic binary string.
Left: two length-$5$ intervals and eight unit intervals for
the periodic pattern $10100101$.
Right: two length-$5$ intervals and eight unit intervals for
the periodic pattern $11000011$.}
\label{fig:w5v2pattern}
\end{figure}

Refer to Figure~\ref{fig:w5v2pattern} for the two cases.
Consider the two length-$5$ intervals containing
the five unit intervals corresponding to the first five bits
and the last five bits, respectively, of this substring.
Since each of them contains three white unit intervals,
they must both be black to avoid forming a white star $K_{1,3}$.
But then the two black length-$5$ intervals, and the two black unit
intervals contained by either of them,
would form a black star $K_{1,3}$.
\end{proof}

The next lemma implies that $\check\kappa(4,2) \ge 3$,
and hence $\check\kappa(4,2) = 3$.

\begin{lemma}\label{lem:w4v2}
Let $\J_4$
be the family of open intervals with integer lengths between $2$ and $7$,
and with integer endpoints between $0$ and $21$.
The interval graph represented by $\J_4$ has claw number $4$
and admits a vertex partition into three subgraphs
with claw number at most $2$,
but it admits no vertex partition into two subgraphs
with claw number at most $2$.
\end{lemma}

\begin{proof}
The interval graph represented by $\J_4$ clearly has claw number $4$,
since the maximum length $7$ is only one plus three times the minimum length $2$.
By the upper bound for the $v=2$ case of Corollary~\ref{cor:kappa'},
we have $\check\kappa(4,2) \le 3$, and hence
the graph represented by $\J_4$ admits a vertex partition into three subgraphs
with claw number at most $2$.
We say that a $2$-partition of a family of intervals is \emph{bad}
if there are four intervals representing a star $K_{1, 3}$ in the same part,
and say that it is \emph{good} otherwise.
Assisted by a computer program,
we next show that $\J_4$ admits no good $2$-partition.

Note that $\J_4$ contains exactly $22 - \ell$ intervals of each length $\ell$,
$2 \le \ell \le 7$.
The total number of intervals in $\J_4$
is $n = 20 + 19 + 18 + 17 + 16 + 15 = 105$.
A brute-force algorithm that enumerates all $2^{105}$ $2$-partitions of $\J_4$
is obviously too slow, but we can speed it up by a standard branch-and-bound technique.

The algorithm works as follows.
Sort the intervals in $\J_4$ by increasing right endpoint,
breaking ties by decreasing left endpoint.
Then call the following recursive function with a partial $2$-partition of $\J_4$
with the first interval in the first part.

Given a good partial $2$-partition of the first $i$ intervals in $\J_4$,
the recursive function reports the $2$-partition and returns if $i = n$.
Otherwise, it checks whether
the next interval in $\J_4$ can be added to either part
without forming a star $K_{1,3}$,
then recurses on the resulting at most two
good partial $2$-partitions of the first $i+1$ intervals.

We wrote a C program implementing this simple branch-and-bound algorithm;
see the source code in the appendix.
On a typical laptop computer,
it took less than two minutes for the program to verify that
$\J_4$ admits no good $2$-partition.
\end{proof}

Recall~\eqref{eq:v2} earlier.
By Lemma~\ref{lem:w6v2},
Lemma~\ref{lem:w5v2},
and
Lemma~\ref{lem:w4v2},
we have $\check\kappa(4,2) = \check\kappa(5,2) = \check\kappa(6,2) = 3$.
The previous three lemmas also imply that
interval graphs admitting no vertex partition into less than three proper
interval graphs do not always contain $H_3$ (which has claw number $9$) as an
induced subgraph.
This partially answers an open question of Gardi~\cite[page~53]{Ga11}
in the negative.

\bigskip
The following lemma shows that $\check\kappa(3,2) \le 2$,
and hence $\check\kappa(3,2) = 2$.

\begin{lemma}\label{lem:w3v2}
Any interval graph with claw number at most $3$ admits a vertex partition into
two subgraphs with claw number at most $2$.
\end{lemma}

\begin{proof}
Let $G$ be an interval graph,
and let $\I$ be an interval representation of $G$ with integer endpoints.
Let $j = \vartheta(G)$.
Run the sweepline algorithm to obtain an independent set of $j$ intervals $T_i$
in $\I$,
and a vertex partition of $\I$ into $j$ cliques $\I'_i$,
where $T_i \in \I'_i$ for $1 \le i \le j$.
For $h \in \{0,1,2,3\}$,
let $\J_h$ be the union of the subfamilies $\I'_i$
with $i \bmod 4 = h$.
We next show that the subgraph of $G$ induced by $\J_0 \cup \J_1$ has
claw number at most $2$. The argument for $\J_2 \cup \J_3$ is similar.

Suppose for contradiction that $\J_0 \cup \J_1$ contains four intervals
$O,A,B,C$ representing a star $K_{1,3}$,
where $O$ is the center, and $A,B,C$ are the three leaves.
Suppose that $O \in \I'_o$, $A \in \I'_a$, $B \in \I'_b$, and $C \in \I'_c$,
where $a < b < c$.
Note there are only two different values $0$ and $1$ for $o,a,b,c$ modulo $4$.
We claim that
if $a > o - 3$, then
$c \ge o + 3$.
Consider two cases.
If $o \bmod 4 = 0$,
then the condition $a > o - 3$ implies that $a \ge o$,
and hence $b \ge o + 1$, and hence $c \ge o + 4$.
If $o \bmod 4 = 1$,
then the condition $a > o - 3$ implies that $a \ge o - 1$,
and hence $b \ge o$, and hence $c \ge o + 3$.

In summary, we must have either $a \le o - 3$ or $c \ge o + 3$:
\begin{itemize}\setlength\itemsep{0pt}

\item
If $a \le o - 3$,
then the $o-a$ intervals $T_i$ for $a \le i \le o - 1$ and the interval $O$
would together be $o-a+1 \ge 4$ pairwise-disjoint intervals intersecting $A$.

\item
If $c \ge o + 3$,
then the $c-o$ intervals $T_i$ for $o \le i \le c - 1$ and the interval $C$
would together be $c-o+1 \ge 4$ pairwise-disjoint intervals intersecting $O$.

\end{itemize}
In both cases, we reach a contradiction to
our assumption that $G$ has claw number at most $3$. 
\end{proof}

We finally prove that $\check\kappa(5,3) = 2$.
The lower bound of $\check\kappa(5,3) \ge 2$ is obvious.
The following lemma gives the tight upper bound of $\check\kappa(5,3) \le 2$:

\begin{lemma}\label{lem:w5v3}
Any interval graph with claw number at most $5$ admits a vertex partition into
two subgraphs with claw number at most $3$.
\end{lemma}

\begin{proof}
Let $G$ be an interval graph with claw number at most $5$,
and let $\I$ be an interval representation of $G$.
Partition $\I$ into two subfamilies $\I_{< 2}$ and $\I_{\ge 2}$ of intervals
properly containing less than two and at least two, respectively,
disjoint intervals in $\I$.
Since any interval that intersects more than three disjoint intervals must properly contain
at least two of them,
the subgraph of $G$ induced by intervals in $\I_{< 2}$ clearly
has claw number at most $3$.
We claim that
the subgraph of $G$ induced by intervals in $\I_{\ge 2}$
has claw number at most $3$ too.

Suppose the contrary. Let $C$ and $L_1,L_2,L_3,L_4$ be intervals in
$\I_{\ge 2}$ representing the center and the four leaves of a star $K_{1,4}$,
where $L_1,L_2,L_3,L_4$ are disjoint and ordered from left to right.
Then $C$ properly contains $L_2$ and $L_3$.
Since each of $L_2$ and $L_3$ also properly contains at least two disjoint
intervals in $\I$, $C$ intersects at least six disjoint intervals in $\I$,
contradicting our assumption that $G$ has claw number at most~$5$.
\end{proof}

The proof of Theorem~\ref{thm:small} is now complete.

\section{Approximation algorithm}

In this section we prove Theorem~\ref{thm:approx}.

Albertson et~al.~\cite[Theorem~4]{AJHL89} showed that
$\check\kappa(w,1) \le w$ for any $w \ge 1$,
which implies that $\check\kappa(v+2,v) = \check\kappa(3,1) \le 3$ when $v = 1$.
When $v = 2$,
by our upper bound of
$\check\kappa(w,v) \le \lfloor \log_{v+1} 3(w-1)/2 \rfloor + 2$
for $w > v = 2$ in Corollary~\ref{cor:kappa'},
we have $\check\kappa(v+2,v) = \check\kappa(4,2) \le \lfloor \log_3 9/2 \rfloor + 2 = 3$.
When $v = 3$,
we have
$\check\kappa(v+2,v) = \check\kappa(5,3) \le 2$
by Lemma~\ref{lem:w5v3}.
When $v > 3$,
we have
$\check\kappa(v+2,v) \le \lfloor\log_{v+1} (v+1)/(v-2)\rfloor + 2 = 2$
by Corollary~\ref{cor:kappa'}.
In summary,
we have $\check\kappa(v+2,v) \le 3$ when $1 \le v \le 2$,
and $\check\kappa(v+2,v) \le 2$ when $v \ge 3$.
Since the proof of Albertson et~al.~\cite[Theorem~4]{AJHL89}
and our proofs of these upper bounds are constructive,
we have the following proposition:

\begin{proposition}\label{prp:v+2}
For any $v \ge 1$,
there is a polynomial-time algorithm that partitions
any interval graph with claw number at most $v+2$ into 
$t$ induced subgraphs with claw number at most $v$,
with $t = 3$ for $1 \le v \le 2$,
and $t = 2$ for $v \ge 3$.
\end{proposition}

Our approximation algorithm 
for \textsc{Min-Partition$(v)$} works as follows.
Given an interval graph $G$,
first obtain an interval representation $\I$ of $G$.
Initialize $\J \gets \I$ and $i \gets 1$.
While $\J$ is not empty,
let $\J'$ be the subfamily of intervals in $\J$
each properly containing at least $v+1$ disjoint intervals in $\J$,
let
$\I_i \gets \J\setminus\J'$,
then update
$\J \gets \J'$ and $i \gets i + 1$.

Let $k$ be the maximum round $i$ in which $\J$ is not empty.
Then $(\I_1,\ldots,\I_k)$ is a $k$-partition of $\I$.
For $1 \le i \le k$,
since each interval in $\I_i$ properly contains at most $v$ disjoint intervals
in $\I_i$,
the subgraph of $G$ induced by each subfamily $\I_i$
has claw number at most $v+2$.
Apply the algorithm in Proposition~\ref{prp:v+2} to partition
the subgraph of $G$ induced by each subfamily $\I_i$ into $t$ subgraphs,
with $t = 3$ for $1 \le v \le 2$,
and $t = 2$ for $v \ge 3$.
Then $G$ is partitioned into $kt$ subgraphs with claw number at most $v$.

Observe that for $2 \le i \le k$,
each interval in $\I_i$ properly contains at least $v+1$ disjoint intervals
in $\I_{i-1}$.
Thus $G$ contains $A_{k,v}$ as an induced subgraph,
and it follows by Lemma~\ref{lem:A} that $G$ admits no vertex partition into
less than $k$ induced subgraphs with claw number at most $v$.
Thus our algorithm for
\textsc{Min-Partition$(v)$}
achieves an approximation ratio of $t$,
with $t = 3$ for $1 \le v \le 2$,
and $t = 2$ for $v \ge 3$.
This completes the proof of Theorem~\ref{thm:approx}.

\section{Open questions}

For $k \ge 2$ and $v \ge 2$,
is there a polynomial-time algorithm for \textsc{$k$-Partition$(v)$} in interval graphs?
In particular, for $k = v = 2$,
is there a polynomial-time algorithm that decides whether an interval graph
admits a vertex partition into two proper interval graphs?
Can the gaps between lower and upper bounds in
Theorem~\ref{thm:mu'} and Corollary~\ref{cor:kappa'}
be reduced?

\appendix

\vspace{\stretch1}
\paragraph{A counterexample}

Gardi reported~\cite[Lemma~2.3]{Ga11} that for $t > 1$,
any $K_{1,t}$-free interval graph admits a vertex partition into
$\big\lceil\log_3\frac{3t-3}2\big\rceil$ proper interval graphs.
By substituting $t$ with $w+1$,
this would imply that
$\check\kappa(w,2) \le \lceil\log_3 3w/2 \rceil$ for $w \ge 1$,
in particular,
$\check\kappa(w,2) \le 2$ for $w \le 6$.
But this cannot hold,
since we proved in Theorem~\ref{thm:small} that
$\check\kappa(4,2) = \check\kappa(5,2) = \check\kappa(6,2) = 3$;
see our Lemmas~\ref{lem:w6v2}, \ref{lem:w5v2}, and~\ref{lem:w4v2}.
We refer to Figure~\ref{fig:gardi} for a counterexample that identifies
a flaw in the proof of~\cite[Lemma~2.3]{Ga11}.
This flaw also affects several other results,
including
an upper bound on $\kappa(n,2)$~\cite[Proposition~2.7]{Ga11}.

\begin{figure}[htbp]
\centering\includegraphics{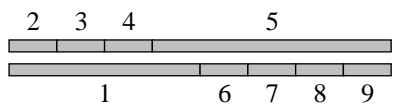}
\caption{A counterexample to the correctness of
the algorithm \textsc{Color-Cliques-Logarithmic}
in Gardi's constructive proof of~\cite[Lemma~2.3]{Ga11} with $t = 6$.
The nine intervals are first partitioned into seven cliques
$C_0 = \{1, 2\}$, $C_1 = \{3\}$, $C_2 = \{4\}$, $C_3 = \{5,6\}$
$C_4 = \{7\}$, $C_5 = \{8\}$, $C_6 = \{9\}$
by the algorithm \textsc{Canonical-Partition-Into-Cliques},
and then grouped into two parts
$\C_0 = \{3,4,7,8\}$ and $\C_1 = \{1,2,5,6,9\}$
by the algorithm \textsc{Color-Cliques-Logarithmic}.
The subgraph represented by $\C_1$ contains an induced star $K_{1,3}$
with center $5$ and three leaves $1,6,9$,
and hence is not a proper interval graph as claimed.}
\label{fig:gardi}
\end{figure}




\newpage
\section*{Source code for Lemma~\ref{lem:w4v2}}

\lstset{basicstyle=\footnotesize\ttfamily,
keywordstyle=\ttfamily,
showstringspaces=false,
tabsize=4}
\begin{lstlisting}
#include <stdio.h>

#define N	105
#define K	512

int ll[N], rr[N], xx[N][K], yy[N][K], zz[N][K], kk[N];
char part[N];

int min(int a, int b) { return a < b ? a : b; }
int max(int a, int b) { return a > b ? a : b; }

int a(int i, int j) { return max(ll[i], ll[j]) < min(rr[i], rr[j]); }

int claw(int i, int j, int k, int l) {
	return a(i, j) && a(i, k) && a(i, l) && !a(j, k) && !a(j, l) && !a(k, l);
}

int bad(int i, int j, int k, int l) {
	return claw(i, j, k, l) || claw(j, k, l, i) || claw(k, l, i, j) || claw(l, i, j, k);
}

int good(int i) {
	int h, p = part[i];

	for (h = 0; h < kk[i]; h++)
		if (part[xx[i][h]] == p && part[yy[i][h]] == p && part[zz[i][h]] == p)
			return 0;
	return 1;
}

int solve(int i) {
	if (i == N) {
		printf("good\n");
		return 1;
	}
	for (part[i] = 0; part[i] <= 1; part[i]++)
		if (good(i) && solve(i + 1))
			return 1;
	return 0;
}

int main() {
	int n, l, r, i;

	n = 0;
	for (r = 2; r <= 21; r++)
		for (l = r - 2; l >= 0 && l >= r - 7; l--)
			ll[n] = l, rr[n] = r, n++;
	printf("n = %d\n", n);

	for (i = 0; i < n; i++) {
		int x, y, z, k = 0;

		for (x = 0; x < i; x++)
			for (y = x + 1; y < i; y++)
				for (z = y + 1; z < i; z++)
					if (bad(i, x, y, z)) {
						if (k == K) {
							printf("K is too small\n");
							return 0;
						}
						xx[i][k] = x, yy[i][k] = y, zz[i][k] = z, k++;
					}
		kk[i] = k;
	}

	if (!solve(1))
		printf("no good\n");
	return 0;
}
\end{lstlisting}

\end{document}